\documentclass[oneside,english]{amsart}
\usepackage[T1]{fontenc}
\usepackage[latin9]{inputenc}
\usepackage{geometry}
\usepackage{textcomp}
\usepackage{amstext}
\usepackage{amsthm}
\usepackage{amssymb}

\makeatletter
\numberwithin{figure}{section}
\theoremstyle{plain}
\newtheorem{thm}{\protect\theoremname}
  \theoremstyle{plain}
  \newtheorem{lem}[thm]{\protect\lemmaname}

\makeatother

\usepackage{babel}
  \providecommand{\lemmaname}{Lemma}
\providecommand{\theoremname}{Theorem}

\begin{document}
\pagestyle{plain}

\title{ON THE CES\`ARO AVERAGE OF THE NUMBERS THAT CAN BE WRITTEN AS SUM OF
A PRIME AND TWO SQUARES OF PRIMES}

\author{Marco Cantarini}
\begin{abstract}
Let $\Lambda\left(n\right)$ be the Von Mangoldt function and $r_{SP}\left(n\right)=\sum_{m_{1}+m_{2}^{2}+m_{3}^{2}=n}\Lambda\left(m_{1}\right)\Lambda\left(m_{2}\right)\Lambda\left(m_{3}\right)$
be the counting function for the numbers that can be written as sum
of a prime and two squares. Let $N$ be a sufficiently large integer.
We prove that 
\[
\sum_{n\leq N}r_{SP}\left(n\right)\frac{\left(N-n\right)^{k}}{\Gamma\left(k+1\right)}=\frac{N^{k+2}\pi}{4\Gamma\left(k+3\right)}+E\left(N,k\right)
\]
for $k>3/2$, where $E\left(N,k\right)$ consists of lower order terms
that are given in terms of $k$ and sum over the non-trivial zeros
of the Riemann zeta function.
\end{abstract}

\maketitle

\section{Introduction}

We continue the recent work of Languasco, Zaccagnini and the author
on additive problems with prime summands. In \cite{langzac1} and
\cite{langzac2} Languasco and Zaccagnini study the Ces\`aro weighted
explicit formula for the Goldbach numbers (the integers that can be
written as sum of two primes) and for the Hardy-Littlewood numbers
(the integers that can be written as sum of a prime and a square).
Recently \cite{Canta} the author wrote a paper regarding the Ces\`aro
average of the integers that can be written as sum of a prime and
two squares. In a similar manner, we will study a Ces\`aro weighted
explicit formula for the integers that can be written as sum of a
prime and two squares of primes. We will obtain an asymptotic formula
with a main term and more terms depending explicitly on the zeros
of the Riemann zeta function. This technique allow us to obtain a
large number of terms in our asymptotic but unfortunately the bound
$k>3/2$ seems to be very difficult to improve. We recall that, for
$k=0$, the Ces\`aro weights vanish so a result for $k\geq0$ would
allow us to get an asymptotic for the mean of $r_{SP}\left(n\right).$
\[
\,
\]

\subjclass 2010 Mathematics Subject Classification:{ Primary 11P32; Secondary 44A10, 33C10}


\newpage{}

\begin{flushleft}
We let
\[
r_{SP}\left(n\right)=\sum_{m_{1}+m_{2}^{2}+m_{3}^{2}=n}\Lambda\left(m_{1}\right)\Lambda\left(m_{2}\right)\Lambda\left(m_{3}\right)
\]
where $\Lambda\left(n\right)$ is the Von Mangoldt function and 
\begin{align}
M_{1}\left(N,k\right)= & \frac{N^{k+2}\pi}{4\Gamma\left(k+3\right)}\label{eq:M1}\\
M_{2}\left(N,k\right)= & \frac{N^{k+1}\pi}{4}\sum_{\rho}\frac{N^{\rho}\Gamma\left(\rho\right)}{\Gamma\left(k+2+\rho\right)}-\frac{N^{k+3/2}\sqrt{\pi}}{2}\sum_{\rho}N^{\rho/2}\frac{\Gamma\left(\rho/2\right)}{\Gamma\left(k+5/2+\rho/2\right)}\label{eq:M2}\\
M_{3}\left(N,k\right)= & \frac{N^{k+1/2}\sqrt{\pi}}{2}\sum_{\rho_{1}}\sum_{\rho_{2}}N^{\rho_{1}+\rho_{2}/2}\frac{\Gamma\left(\rho_{1}\right)\Gamma\left(\frac{\rho_{2}}{2}\right)}{\Gamma\left(k+3/2+\rho_{1}+\rho_{2}/2\right)}\nonumber \\
+ & \frac{N^{k+1}}{4}\sum_{\rho_{1}}\sum_{\rho_{2}}N^{\rho_{1}/2+\rho_{2}/2}\frac{\Gamma\left(\frac{\rho_{1}}{2}\right)\Gamma\left(\frac{\rho_{2}}{2}\right)}{\Gamma\left(k+2+\rho_{1}/2+\rho_{2}/2\right)}\label{eq:M3}\\
M_{4}\left(N,k\right)= & \frac{N^{k}}{4}\sum_{\rho_{1}}\sum_{\rho_{2}}\sum_{\rho_{3}}\frac{N^{\rho_{1}+\rho_{2}/2+\rho_{3}/2}\Gamma\left(\rho_{1}\right)\Gamma\left(\frac{\rho_{2}}{2}\right)\Gamma\left(\frac{\rho_{3}}{2}\right)}{\Gamma\left(k+\rho_{1}+\rho_{2}/2+\rho_{3}/2\right)}.\label{eq:M4}
\end{align}
The main result of this paper is the following
\par\end{flushleft}
\begin{thm}
Let $N$ be a sufficient large integer. We have
\[
\sum_{n\leq N}r_{SP}\left(n\right)\frac{\left(N-n\right)^{k}}{\Gamma\left(k+1\right)}=M_{1}\left(N,k\right)+M_{2}\left(N,k\right)+M_{3}\left(N,k\right)+M_{4}\left(N,k\right)+O\left(N^{k+1}\right)
\]
for $k>3/2$, where $\rho=\beta+i\gamma$, with or without subscripts,
runs over the non-trivial zeros of the Riemann zeta function $\zeta\left(s\right)$. 
\end{thm}
Note that an upper bound for $M_{i}\left(N,k\right),\,i=2,\dots,4$
depends closely on $\beta$. Let us define
\[
\overline{\beta}:=\sup\left\{ \beta:\,\textrm{Re}\left(\rho\right)=\beta\right\} .
\]
 We have that 
\[
M_{2}\left(N,k\right)\ll_{k}N^{k+3/2+\overline{\beta}/2}
\]
\[
M_{3}\left(N,k\right)\ll_{k}N^{k+1+\overline{\beta}}
\]
\[
M_{4}\left(N,k\right)\ll_{k}N^{k+2\overline{\beta}}.
\]
Note also that, if the Riemann hypothesis is true, then $M_{4}\left(N,k\right)$
can be incorporated in the error term. The problem of representing
an integer as sum of a prime and two prime squares is classical. Let
\[
A=\left\{ n\in\mathbb{N}:\,n\equiv1\,\mod\,2;\,n\not\equiv2\,\mod\,3\right\} ;
\]
 it is conjectured that every sufficiently large natural number $n\in A$
is a sum of a prime and two prime squares. Many authors studied the
cardinality $E\left(N\right)$ of the set of integers $n\leq N$,
$n\in A$ that are not representable as a sum of prime and two squares
of primes. We recall Hua \cite{hua}, Schwarz \cite{schwarz}, Leung-Liu
\cite{LeuLiu}, Wang \cite{wang}, Wang-Meng \cite{wangmeng}, Li
\cite{li}, Harman-Kumchev \cite{harmkum}. Zhao \cite{zhao} proved
that
\[
E\left(N\right)\ll N^{1/3+\epsilon}
\]
and so every integer $n\in\left[1,N\right]\cap A$, with at most $O\left(N^{1/3+\epsilon}\right)$
exceptions, is a sum of a prime and two squares of primes. Letting
\[
r\left(n\right):=\sum_{p_{1}+p_{2}^{2}+p_{3}^{2}=n}\log\left(p_{1}\right)\log\left(p_{2}\right)\log\left(p_{3}\right)
\]
Zhao also proved that an asymptotic formula for $r\left(n\right)$
holds for $n\in\left[1,N\right]\cap A$, with at most $O\left(N^{1/3+\epsilon}\right)$
exceptions. Similar averages of arithmetical functions are common
in literature, see, e.g., Chandrasekharan - Narasimhan \cite{ChaNa}
and Berndt \cite{Berndt} who built on earlier classical work. The
method we will use in this additive problem is based on a formula
due to Laplace \cite{Lap}, namely
\begin{equation}
\frac{1}{2\pi i}\int_{\left(a\right)}v^{-s}e^{v}dv=\frac{1}{\Gamma\left(s\right)}\label{eq:Lap}
\end{equation}
with $\textrm{Re}\left(s\right)>0$ and $a>0$ (see, e.g., formula
5.4 (1) on page 238 of \cite{Erd}), where the notation $\int_{\left(a\right)}$
means $\int_{a-i\infty}^{a+i\infty}$. As in \cite{langzac2}, we
combine this approach with line integrals with the classical methods
dealing with infinite sum over primes and integers.

I thank A. Zaccagnini and A. Languasco for their contributions and
the conversations on this topic. I also thank the referee, who pointed
out further inaccuracies and suggested improvements in the presentation.
This work is part of the Author's Ph.D. thesis.

\section{Preliminary definitions and Lemmas}

Let $z=a+iy$, $a>0$, let
\begin{align}
\widetilde{S}_{1}\left(z\right)= & \sum_{m\geq1}\Lambda\left(m\right)e^{-mz}\label{eq:stild}\\
\widetilde{S}_{2}\left(z\right)= & \sum_{m\geq1}\Lambda\left(m\right)e^{-m^{2}z}\label{eq:stild2}
\end{align}
and let us introduce the following
\begin{lem}
Let $z=a+iy,\,a>0$ and $y\in\mathbb{R}$. Then
\begin{equation}
\widetilde{S}_{1}\left(z\right)=\frac{1}{z}-\sum_{\rho}z^{-\rho}\Gamma\left(\rho\right)+E\left(a,y\right)\label{eq:stilda}
\end{equation}
where $\rho=\beta+i\gamma$ runs over the non-trivial zeros of $\zeta\left(s\right)$
and
\begin{equation}
E\left(a,y\right)\ll\left|z\right|^{1/2}\begin{cases}
1, & \left|y\right|\leq a\\
1+\log^{2}\left(\left|y\right|/a\right), & \left|y\right|>a.
\end{cases}\label{eq:error}
\end{equation}
\end{lem}
(For a proof see Lemma 1 of \cite{langzac1}. The bound for $E\left(a,y\right)$
has been corrected in \cite{Lang}). So in particular, taking $z=\frac{1}{N}+iy$
we have 
\[
\left|\sum_{\rho}z^{-\rho}\Gamma\left(\rho\right)\right|=\left|\frac{1}{z}-\widetilde{S}\left(z\right)+E\left(\frac{1}{N},y\right)\right|\ll N+\frac{1}{\left|z\right|}+\left|E\left(\frac{1}{N},y\right)\right|
\]
\begin{equation}
\ll\begin{cases}
N, & \left|y\right|\leq1/N\\
N+\left|z\right|^{1/2}\log^{2}\left(2N\left|y\right|\right), & \left|y\right|>1/N.
\end{cases}\label{eq:lemmaspec}
\end{equation}

We now introduce the following
\begin{lem}
Let $z=a+iy,\,a>0$, $y\in\mathbb{R}$ and $\ell$ a fixed positive
integer. Then
\begin{equation}
\widetilde{S}_{\ell}\left(z\right)=\frac{\Gamma\left(1/\ell\right)}{\ell z^{1/\ell}}-\frac{1}{\ell}\sum_{\rho}z^{-\rho/\ell}\Gamma\left(\frac{\rho}{\ell}\right)+E_{\ell}\left(a,y\right)\label{eq:stilda-l}
\end{equation}
where $\rho=\beta+i\gamma$ runs over the non-trivial zeros of $\zeta\left(s\right)$
and
\begin{equation}
E_{\ell}\left(a,y\right)\ll_{\ell}E\left(a,y\right).\label{eq:error-1}
\end{equation}
\end{lem}
\begin{proof}
It is well know (see for example formula 5 of \cite{langzac3}) that,
for $\ell\in\mathbb{N}_{0}$,
\begin{align}
\widetilde{S}_{\ell}\left(z\right)= & \sum_{m\geq1}\Lambda\left(m\right)e^{-m^{\ell}z}\nonumber \\
= & \frac{\Gamma\left(1/\ell\right)}{\ell z^{1/\ell}}-\frac{1}{\ell}\sum_{\rho}z^{-\rho/\ell}\Gamma\left(\frac{\rho}{\ell}\right)-\frac{\zeta'}{\zeta}\left(0\right)-\frac{1}{2\pi i}\int_{\left(-1/2\right)}\frac{\zeta'}{\zeta}\left(\ell w\right)\Gamma\left(w\right)z^{-w}dw\label{eq:lem3int}
\end{align}
so, taking $w=-\frac{1}{2}+it$, following the proof of the Lemma
1 in \cite{langzac1} and observing that 
\[
\left|\frac{\zeta'}{\zeta}\left(\ell w\right)\right|\ll_{\ell}\log\left(\left|t\right|+2\right)
\]
we can conclude that we may estimate the integral in (\ref{eq:lem3int})
exactly as in \cite{langzac1}, so the claim follows. 
\end{proof}
Now we have to recall that the Prime Number Theorem (PNT) is equivalent,
via Lemma 2, to the statement
\begin{equation}
\widetilde{S}_{1}\left(a\right)\sim a^{-1},\textrm{ when }a\rightarrow0^{+}\label{eq:asims1}
\end{equation}
(see Lemma 9 of \cite{HarLit}) and from Lemma 3 we have 
\begin{equation}
\widetilde{S}_{2}\left(a\right)\sim\frac{\sqrt{\pi}}{2a^{1/2}},\textrm{ when }a\rightarrow0^{+}.\label{eq:asims2}
\end{equation}
 For our purposes it is important to introduce the Stirling approximation
(see for example \S 4.42 of \cite{titfun})
\begin{equation}
\left|\Gamma\left(x+iy\right)\right|\sim\sqrt{2\pi}e^{-\pi\left|y\right|/2}\left|y\right|^{x-1/2}\label{eq:1}
\end{equation}
 uniformly for $x\in\left[x_{1},x_{2}\right]$, $x_{1}$ and $x_{2}$
fixed, as well as the identity 
\begin{equation}
\left|z^{-w}\right|=\left|z\right|^{-\textrm{Re}\left(w\right)}\exp\left(\textrm{Im}\left(w\right)\arctan\left(y/a\right)\right).\label{eq:modzcomplgen}
\end{equation}

We now quote Lemmas 2 and 3 from \cite{langzac1}:
\begin{lem}
Let $\beta+i\gamma$ run over the non-trivial zeros of the Riemann
zeta function and let $\alpha>1$ be a parameter. The series
\[
\sum_{\rho,\,\gamma>0}\gamma^{\beta-1/2}\int_{1}^{\infty}\exp\left(-\gamma\arctan\left(1/u\right)\right)\frac{du}{u^{\alpha+\beta}}
\]
converges provided that $\alpha>3/2$. For $\alpha\leq3/2$ the series
does not converge. The result remains true if we insert in the integral
a factor $\log^{c}\left(u\right)$, for any fixed $c\geq0$.
\end{lem}
$\ $
\begin{lem}
Let $\beta+i\gamma$ run over the non-trivial zeros of the Riemann
zeta function, let $z=a+iy,$ $a\in\left(0,1\right)$, $y\in\mathbb{R}$
and $\alpha>1$. We have
\[
\sum_{\rho}\left|\gamma\right|^{\beta-1/2}\int_{\mathbb{Y}_{1}\cup\mathbb{Y}_{2}}\exp\left(\gamma\arctan\left(\frac{y}{a}\right)-\frac{\pi}{2}\left|\gamma\right|\right)\frac{dy}{\left|z\right|^{\alpha+\beta}}\ll_{\alpha}a^{-\alpha}
\]
where $\mathbb{Y}_{1}=\left\{ y\in\mathbb{R}:\,\gamma y\leq0\right\} $
and $\mathbb{Y}_{2}=\left\{ y\in\left[-a,a\right]:\,y\gamma>0\right\} $.
The result remains true if we insert in the integral a factor $\log^{c}\left(\left|y\right|/a\right)$,
for any fixed $c\geq0.$
\end{lem}
Let us introduce another lemma
\begin{lem}
Let $\rho=\beta+i\gamma$ run over the non-trivial zeros of the Riemann
zeta function, let $z=\frac{1}{N}+iy,$ where $N>1$ is a natural
number, $y\in\mathbb{R}$, $\ell\geq1$ an integer and $\alpha>3/2$.
We have
\[
\sum_{\rho}\left|\Gamma\left(\frac{\rho}{\ell}\right)\right|\int_{\left(1/N\right)}\left|e^{Nz}\right|\left|z^{-\rho/\ell}\right|\left|z\right|^{-\alpha}\left|dz\right|\ll_{\alpha}N^{\alpha}.
\]
\end{lem}
\begin{proof}
Put $a=\frac{1}{N}.$ Using the identity (\ref{eq:modzcomplgen}),
(\ref{eq:1}) and 
\begin{equation}
\left|z\right|^{-1}\asymp\begin{cases}
a^{-1} & \left|y\right|\leq a,\\
\left|y\right|^{-1} & \left|y\right|\geq a,
\end{cases}\label{eq:modz-1}
\end{equation}
we get that the left hand side in the statement above is

\begin{equation}
\sum_{\rho}\left|\gamma\right|^{\beta/\ell-1/2}\int_{\mathbb{R}}\exp\left(\frac{\gamma}{\ell}\arctan\left(\frac{y}{a}\right)-\frac{\pi}{2}\frac{\left|\gamma\right|}{\ell}\right)\frac{dy}{\left|z\right|^{\alpha+\beta/\ell}}.\label{eq:lemma6}
\end{equation}
The case $\ell=1$ has already been discussed in Lemma 6 of \cite{Canta}.
For $\ell>1$, observing Lemmas 2 and 3 of \cite{langzac1} and Lemma
6 \cite{Canta}, we can conclude that the presence of $\ell$ does
not alter the proofs. Hence using the same argument of Lemma 6 of
\cite{Canta} we have the convergence for $\alpha>3/2$. 
\end{proof}

\section{Setting}

From (\ref{eq:stild}) and (\ref{eq:stild2}) it is not hard to see
that 
\[
\widetilde{S}_{1}\left(z\right)\widetilde{S}_{2}^{2}\left(z\right)=\sum_{m_{1}\geq1}\sum_{m_{2}\geq1}\sum_{m_{3}\geq1}\Lambda\left(m_{1}\right)\Lambda\left(m_{2}\right)\Lambda\left(m_{3}\right)e^{-\left(m_{1}+m_{2}^{2}+m_{3}^{2}\right)z}=\sum_{n\geq1}r_{SP}\left(n\right)e^{-nz}
\]
so let $z=a+iy$ and $a>0$ and let us consider
\[
\frac{1}{2\pi i}\int_{\left(a\right)}e^{Nz}z^{-k-1}\widetilde{S}_{1}\left(z\right)\widetilde{S}_{2}^{2}\left(z\right)dz=\frac{1}{2\pi i}\int_{\left(a\right)}e^{Nz}z^{-k-1}\sum_{n\geq1}r_{SP}\left(n\right)e^{-nz}dz.
\]
Now we prove that we can exchange the integral with the series. From
(\ref{eq:asims1}) and (\ref{eq:asims2}) we have 
\[
\sum_{n\geq1}\left|r_{SP}\left(n\right)e^{-nz}\right|=\widetilde{S}_{1}\left(a\right)\widetilde{S}_{2}^{2}\left(a\right)\ll a^{-2}
\]
hence
\begin{align*}
\int_{\left(a\right)}\left|e^{Nz}z^{-k-1}\right|\left|\widetilde{S}_{1}\left(z\right)\widetilde{S}_{2}^{2}\left(z\right)\right|\left|dz\right|\ll & a^{-2}e^{Na}\left(\int_{-a}^{a}a^{-k-1}dy+2\int_{a}^{\infty}y^{-k-1}dy\right)\\
\ll_{k} & a^{-2-k}e^{Na}
\end{align*}
assuming $k>0$, so we have that 
\begin{equation}
\sum_{n\leq N}r_{SP}\left(n\right)\frac{\left(N-n\right)^{k}}{\Gamma\left(k+1\right)}=\frac{1}{2\pi i}\int_{\left(a\right)}e^{Nz}z^{-k-1}\widetilde{S}_{1}\left(z\right)\widetilde{S}_{2}^{2}\left(z\right)dz.\label{eq:main}
\end{equation}
Now from (\ref{eq:stilda}), (\ref{eq:stilda-l}), (\ref{eq:asims1})
and (\ref{eq:asims2}) and observing that, for $\ell\geq1,$ 
\[
\frac{\Gamma\left(1/\ell\right)}{\ell z^{1/\ell}}-\frac{1}{\ell}\sum_{\rho}z^{-\rho/\ell}\Gamma\left(\frac{\rho}{\ell}\right)=\widetilde{S}_{\ell}\left(z\right)-E_{\ell}\left(a,y\right)\ll a^{-1/\ell}+\left|E_{\ell}\left(a,y\right)\right|
\]
we have
\begin{align}
\widetilde{S}_{1}\left(z\right)\widetilde{S}_{2}^{2}\left(z\right)= & \left(\frac{1}{z}-\sum_{\rho}z^{-\rho}\Gamma\left(\rho\right)+E_{1}\left(a,y\right)\right)\left(\frac{\sqrt{\pi}}{2z^{1/2}}-\frac{1}{2}\sum_{\rho}z^{-\rho/2}\Gamma\left(\frac{\rho}{2}\right)+E_{2}\left(a,y\right)\right)^{2}\nonumber \\
= & \left(\frac{1}{z}-\sum_{\rho}z^{-\rho}\Gamma\left(\rho\right)\right)\left(\frac{\sqrt{\pi}}{2z^{1/2}}-\frac{1}{2}\sum_{\rho}z^{-\rho/2}\Gamma\left(\frac{\rho}{2}\right)\right)^{2}\nonumber \\
+ & O\left(\left|E_{1}\left(a,y\right)\right|a^{-1}+\left|E_{1}\left(a,y\right)\right|\left|E_{2}\left(a,y\right)\right|^{2}+\left|E_{1}\left(a,y\right)\right|\left|E_{2}\left(a,y\right)\right|a^{-1/2}\right.\label{eq:comberrors0}\\
+ & \left.\left|E_{2}\left(a,y\right)\right|^{2}a^{-1}+a^{-3/2}\left|E_{2}\left(a,y\right)\right|\right).\label{eq:comberrors}
\end{align}
Now let us consider $l,m,r,s\geq1$ integers. From (\ref{eq:error})
and (\ref{eq:error-1}) we have that 
\[
\int_{\left(a\right)}\left|e^{Nz}z^{-k-1}\right|\left|E_{l}\left(a,y\right)\right|^{r}\left|E_{m}\left(a,y\right)\right|^{s}\left|dz\right|
\]
\[
\ll_{l,m,r,s}e^{Na}\left(a^{-k-1+\frac{r+s}{2}}\int_{0}^{a}dy+\int_{a}^{\infty}y^{-k-1+\frac{r+s}{2}}\log^{2r+2s}\left(\frac{y}{a}\right)dy\right)
\]
\[
\ll_{l,m,r,s}e^{Na}a^{-k+\frac{r+s}{2}}
\]
assuming $k>\frac{r+s}{2}$. We now have to deal with the terms in
(\ref{eq:comberrors0}) and (\ref{eq:comberrors}): taking $a=1/N$
we can observe that 
\[
\int_{\left(1/N\right)}\left|e^{Nz}z^{-k-1}\right|\left|E_{2}\left(1/N,y\right)\right|^{2}\left|E_{1}\left(1/N,y\right)\right|\left|dz\right|\ll_{k}N^{k-3/2},
\]
\[
N^{1/2}\int_{\left(1/N\right)}\left|e^{Nz}z^{-k-1}\right|\left|E_{2}\left(1/N,y\right)\right|\left|E_{1}\left(1/N,y\right)\right|\left|dz\right|\ll N^{k-1/2},
\]
\[
N\int_{\left(1/N\right)}\left|e^{Nz}z^{-k-1}\right|\left|E_{2}\left(1/N,y\right)\right|^{2}\left|dz\right|\ll N^{k}
\]
\[
N^{3/2}\int_{\left(1/N\right)}\left|e^{Nz}z^{-k-1}\right|\left|E_{2}\left(1/N,y\right)\right|\left|dz\right|\ll N^{k+1}
\]
and
\[
N\int_{\left(1/N\right)}\left|e^{Nz}z^{-k-1}\right|\left|E_{1}\left(1/N,y\right)\right|\left|dz\right|\ll N^{k+1/2};
\]
hence the Ces\`aro average of $r_{SP}\left(n\right)$ can be broken
down as
\begin{align*}
\sum_{n\leq N}r_{SP}\left(n\right)\frac{\left(N-n\right)^{k}}{\Gamma\left(k+1\right)}= & \frac{1}{2\pi i}\int_{\left(1/N\right)}e^{Nz}z^{-k-1}\left(\frac{1}{z}-\sum_{\rho}z^{-\rho}\Gamma\left(\rho\right)\right)\left(\frac{\sqrt{\pi}}{2z^{1/2}}-\frac{1}{2}\sum_{\rho}z^{-\rho/2}\Gamma\left(\frac{\rho}{2}\right)\right)^{2}dz\\
+ & O_{k}\left(N^{k+1}\right)\\
= & \frac{1}{8i}\int_{\left(1/N\right)}e^{Nz}z^{-k-3}dz\\
+ & \frac{1}{8i}\int_{\left(1/N\right)}e^{Nz}z^{-k-2}\sum_{\rho}z^{-\rho}\Gamma\left(\rho\right)dz\\
- & \frac{1}{4\sqrt{\pi}i}\int_{\left(1/N\right)}e^{Nz}z^{-k-5/2}\sum_{\rho}z^{-\rho/2}\Gamma\left(\frac{\rho}{2}\right)dz\\
+ & \frac{1}{4\sqrt{\pi}i}\int_{\left(1/N\right)}e^{Nz}z^{-k-3/2}\sum_{\rho_{1}}z^{-\rho_{1}}\Gamma\left(\rho_{1}\right)\sum_{\rho_{2}}z^{-\rho_{2}/2}\Gamma\left(\frac{\rho_{2}}{2}\right)dz\\
+ & \frac{1}{8\pi i}\int_{\left(1/N\right)}e^{Nz}z^{-k-2}\sum_{\rho_{1}}z^{-\rho_{1}/2}\Gamma\left(\frac{\rho_{1}}{2}\right)\sum_{\rho_{2}}z^{-\rho_{2}/2}\Gamma\left(\frac{\rho_{2}}{2}\right)dz\\
- & \frac{1}{8\pi i}\int_{\left(1/N\right)}e^{Nz}z^{-k-1}\sum_{\rho_{1}}z^{-\rho_{1}}\Gamma\left(\rho_{1}\right)\sum_{\rho_{2}}z^{-\rho_{2}/2}\Gamma\left(\frac{\rho_{2}}{2}\right)\sum_{\rho_{3}}z^{-\rho_{3}/2}\Gamma\left(\frac{\rho_{3}}{2}\right)dz\\
+ & O_{k}\left(N^{k+1}\right)\\
=: & I_{1}+I_{2}+I_{3}+I_{4}+I_{5}+I_{6}+O_{k}\left(N^{k+1}\right),
\end{align*}
say. In the next sections we will prove that $I_{1}=M_{1}\left(N,k\right)$,
$I_{2}+I_{3}=M_{2}\left(N,k\right)$, $I_{4}+I_{5}=M_{3}\left(N,k\right)$
and $I_{6}=M_{4}\left(N,k\right).$

\section{Evaluation of $I_{1}$}

From $I_{1}$ we will find the main term. If we put $Nz=s$ we get
\begin{equation}
I_{1}=\frac{1}{8i}\int_{\left(1/N\right)}e^{Nz}z^{-k-3}dz=\frac{N^{k+2}}{8i}\int_{\left(1\right)}e^{s}s^{-k-3}ds=\frac{N^{k+2}\pi}{4\Gamma\left(k+3\right)}
\end{equation}
using (\ref{eq:Lap}). Then $I_{1}=M_{1}\left(N,k\right)$.

\section{Evaluation of $I_{2}$ and $I_{3}$}

We have
\[
I_{2}=\frac{1}{8i}\int_{\left(1/N\right)}e^{Nz}z^{-k-2}\sum_{\rho}z^{-\rho}\Gamma\left(\rho\right)dz
\]
 and 
\[
I_{3}=-\frac{1}{4\sqrt{\pi}i}\int_{\left(1/N\right)}e^{Nz}z^{-k-5/2}\sum_{\rho}z^{-\rho/2}\Gamma\left(\frac{\rho}{2}\right)dz.
\]
We want to exchange the integral with the series, then we will prove
the absolute convergence for a suitable choice of $k.$ Hence we have
to study
\[
A_{2}:=\left|\sum_{\rho}\Gamma\left(\rho\right)\right|\int_{\left(1/N\right)}\left|e^{Nz}\right|\left|z^{-k-2}\right|\left|z^{-\rho}\right|\left|dz\right|
\]
 and
\[
A_{3}:=\left|\sum_{\rho}\Gamma\left(\frac{\rho}{2}\right)\right|\int_{\left(1/N\right)}\left|e^{Nz}\right|\left|z^{-k-5/2}\right|\left|z^{-\rho/2}\right|\left|dz\right|
\]
 and from Lemma 6 we have the convergence for $k>-1/2$ and $k>-1$
respectively. So we can switch the integral and the series and get
\[
I_{2}=\frac{1}{8i}\sum_{\rho}\Gamma\left(\rho\right)\int_{\left(1/N\right)}e^{Nz}z^{-k-2-\rho}dz=\frac{N^{k+1}\pi}{4}\sum_{\rho}N^{\rho}\frac{\Gamma\left(\rho\right)}{\Gamma\left(k+2+\rho\right)}
\]
and 
\[
I_{3}=-\frac{1}{4\sqrt{\pi}i}\sum_{\rho}\Gamma\left(\frac{\rho}{2}\right)\int_{\left(1/N\right)}e^{Nz}z^{-k-5/2-\rho/2}dz=-\frac{N^{k+3/2}\sqrt{\pi}}{2}\sum_{\rho}N^{\rho/2}\frac{\Gamma\left(\rho/2\right)}{\Gamma\left(k+5/2+\rho/2\right)}
\]
then $I_{2}+I_{3}=M_{2}\left(N,k\right)$.

\section{Evaluation of $I_{4}$}

We have to evaluate 
\[
I_{4}=\frac{1}{4\sqrt{\pi}i}\int_{\left(1/N\right)}e^{Nz}z^{-k-3/2}\sum_{\rho_{1}}z^{-\rho_{1}}\Gamma\left(\rho_{1}\right)\sum_{\rho_{2}}z^{-\rho_{2}/2}\Gamma\left(\frac{\rho_{2}}{2}\right)dz.
\]
We want to switch the integral with two series so we will prove the
absolute convergence of 
\[
A_{4,1}:=\sum_{\rho_{1}}\left|\Gamma\left(\rho_{1}\right)\right|\int_{\left(1/N\right)}\left|e^{Nz}\right|\left|z^{-k-3/2}\right|\left|z^{-\rho_{1}}\right|\left|\sum_{\rho_{2}}z^{-\rho_{2}/2}\Gamma\left(\frac{\rho_{2}}{2}\right)\right|\left|dz\right|
\]
and
\[
A_{4,2}:=\sum_{\rho_{1}}\left|\Gamma\left(\rho_{1}\right)\right|\sum_{\rho_{2}}\left|\Gamma\left(\frac{\rho_{2}}{2}\right)\right|\int_{\left(1/N\right)}\left|e^{Nz}\right|\left|z^{-k-3/2}\right|\left|z^{-\rho_{1}}\right|\left|z^{-\rho_{2}/2}\right|\left|dz\right|.
\]
Now we have to introduce some notations, which is necessary since
the evaluation of the integrals depends strictly on the sign of $y$
and the sign of the imaginary part of $\rho.$ Assume that $A_{m,n}:=\int_{\left(1/N\right)}\dots\left|dz\right|.$
Hereafter we will use the symbol 
\begin{equation}
A_{m,n}^{+}:=\int_{0}^{1/N+i\infty}\dots\left|dz\right|\label{eq:simbol+}
\end{equation}
and 
\begin{equation}
A_{m,n}^{-}:=\int_{1/N-i\infty}^{0}\dots\left|dz\right|.\label{eq:simbol-}
\end{equation}
 From (\ref{eq:stilda-l}) we can see that 
\[
\left|\sum_{\rho_{2}}z^{-\rho_{2}/2}\Gamma\left(\frac{\rho_{2}}{2}\right)\right|=\left|\widetilde{S}_{2}\left(z\right)-\frac{\sqrt{\pi}}{2z^{1/2}}-E_{2}\left(1/N,y\right)\right|\ll N^{1/2}+\frac{1}{\left|z\right|^{1/2}}+\left|E_{2}\left(1/N,y\right)\right|
\]
\begin{equation}
\ll\begin{cases}
N, & \left|y\right|\leq1/N\\
N+\left|z\right|^{1/2}\log^{2}\left(2N\left|y\right|\right), & \left|y\right|>1/N.
\end{cases}\label{eq:cambioeval}
\end{equation}
Let us consider $y\leq0$ and, recalling the notation $\rho_{j}=\beta_{j}+i\gamma_{j}$,
the notation (\ref{eq:simbol-}) and assuming $\gamma_{1}>0$ for
symmetry, we have to study 
\[
A_{4,1}^{-}\ll\sum_{\rho_{1:\,}\gamma_{1}>0}\gamma_{1}^{\beta_{1}-1/2}\left(N\int_{-1/N}^{0}\frac{\exp\left(\gamma_{1}\arctan\left(Ny\right)-\frac{\pi}{2}\gamma_{1}\right)}{\left|z\right|^{k+3/2+\beta_{1}}}\left|dz\right|+\right.
\]
\[
\left.N\int_{-\infty}^{-1/N}\frac{\exp\left(\gamma_{1}\arctan\left(Ny\right)-\frac{\pi}{2}\gamma_{1}\right)}{\left|y\right|^{k+3/2+\beta_{1}}}dy+\int_{-\infty}^{-1/N}\frac{\exp\left(\gamma_{1}\arctan\left(Ny\right)-\frac{\pi}{2}\gamma_{1}\right)\log^{2}\left(2N\left|y\right|\right)}{\left|y\right|^{k+1+\beta_{1}}}dy\right)\ll_{k}N^{k+5/2}
\]
from Lemma 5, assuming that $k>0.$ Note that we have to split the
integral since, from (\ref{eq:modz-1}) and (\ref{eq:cambioeval}),
we have different evaluations if $\left|y\right|\leq1/N$ or $\left|y\right|>1/N.$
Now let us consider $y>0$. Recalling (\ref{eq:simbol+}), we have
to study 
\[
A_{4,1}^{+}\ll N\sum_{\rho_{1}:\,\gamma_{1}>0}\gamma_{1}^{\beta_{1}-1/2}\int_{0}^{1/N}\frac{\exp\left(\gamma_{1}\arctan\left(Ny\right)-\frac{\pi}{2}\gamma_{1}\right)}{\left|z\right|^{k+3/2+\beta_{1}}}\left|dz\right|
\]
\[
+\sum_{\rho_{1}:\,\gamma_{1}>0}\gamma_{1}^{\beta_{1}-1/2}\int_{1/N}^{\infty}\exp\left(\gamma_{1}\arctan\left(Ny\right)-\frac{\pi}{2}\gamma_{1}\right)\frac{N+y^{1/2}\log^{2}\left(2Ny\right)}{y^{k+3/2+\beta_{1}}}dy=\mathcal{A}_{1}+\mathcal{A}_{2}
\]
say, and we have that 
\[
\mathcal{A}_{1}\ll_{k}N^{k+5/2}
\]
from Lemma 5 and 
\[
\mathcal{A}_{2}\ll N^{k+5/2}\sum_{\rho_{1}:\,\gamma_{1}>0}\gamma_{1}^{\beta_{1}-1/2}\int_{1}^{\infty}\frac{\exp\left(-\gamma_{1}\arctan\left(1/u\right)\right)}{u^{k+3/2+\beta_{1}}}du
\]
\[
+N^{k+1}\sum_{\rho_{1}:\,\gamma_{1}>0}\gamma_{1}^{\beta_{1}-1/2}\int_{1}^{\infty}\frac{\exp\left(-\gamma_{1}\arctan\left(1/u\right)\right)\log^{2}\left(2u\right)}{u^{k+1/2+\beta_{1}}}du\ll_{k}N^{k+5/2}
\]
from Lemma 4, assuming $k>1/2$. Now let us consider 
\[
A_{4,2}=\sum_{\rho_{1}}\left|\Gamma\left(\rho_{1}\right)\right|\sum_{\rho_{2}}\left|\Gamma\left(\frac{\rho_{2}}{2}\right)\right|\int_{\left(1/N\right)}\left|e^{Nz}\right|\left|z^{-k-3/2}\right|\left|z^{-\rho_{1}}\right|\left|z^{-\rho_{2}/2}\right|\left|dz\right|.
\]
By symmetry, it suffices to consider only the cases $\gamma_{1},\gamma_{2}>0$
and $\gamma_{1}>0,$ $\gamma_{2}<0$. As in (\ref{eq:simbol+}) and
(\ref{eq:simbol-}) we have to introduce some new notations since
the evaluation depends on the sign of the product $\gamma_{1}\gamma_{2}$
and the sign of $y.$ Hereafter we will use the symbol $B_{m,n}$
when we consider $A_{m,n}$ with the assumption $\gamma_{1},\gamma_{2}>0$
and the symbol $C_{m,n}$ when we consider $A_{m,n}$ with the assumption
$\gamma_{1}>0,$ $\gamma_{2}<0$. Since 
\begin{equation}
\arctan\left(Ny\right)-\frac{\pi}{2}\leq-\frac{\pi}{2}\label{eq:arctanboun1}
\end{equation}
 and recalling (\ref{eq:simbol-}), we have 
\[
B_{4,2}^{-}\ll\sum_{\rho_{1}:\,\gamma_{1}>0}\gamma_{1}^{\beta_{1}-1/2}\exp\left(-\frac{\pi}{2}\gamma_{1}\right)\sum_{\rho_{2}:\,\gamma_{2}>0}\gamma_{2}^{\beta_{2}/2-1/2}\exp\left(-\frac{\pi}{4}\gamma_{2}\right)\left(\int_{-\infty}^{0}\frac{dy}{\left|z\right|^{k+3/2+\beta_{1}+\beta_{2}/2}}\right)
\]
\[
\ll_{k}N^{k+2}\sum_{\rho_{1}:\,\gamma_{1}>0}\gamma_{1}^{\beta_{1}-1/2}\exp\left(-\frac{\pi}{2}\gamma_{1}\right)\sum_{\rho_{2}:\,\gamma_{2}>0}\gamma_{2}^{\beta_{2}/2-1/2}\exp\left(-\frac{\pi}{4}\gamma_{2}\right)\ll_{k}N^{k+2}.
\]
Now let us consider $y>0$.We have 
\[
B_{4,2}^{+}\ll\sum_{\rho_{1}:\,\gamma_{1}>0}\gamma_{1}^{\beta_{1}-1/2}\sum_{\rho_{2}:\,\gamma_{2}>0}\gamma_{2}^{\beta_{2}/2-1/2}\int_{0}^{1/N}\frac{\exp\left(\left(\gamma_{1}+\frac{\gamma_{2}}{2}\right)\left(\arctan\left(Ny\right)-\frac{\pi}{2}\right)\right)}{\left|z\right|^{k+3/2+\beta_{1}+\beta_{2}/2}}dy
\]
\[
+\sum_{\rho_{1}:\,\gamma_{1}>0}\gamma_{1}^{\beta_{1}-1/2}\sum_{\rho_{2}:\,\gamma_{2}>0}\gamma_{2}^{\beta_{2}/2-1/2}\int_{1/N}^{\infty}\frac{\exp\left(\left(\gamma_{1}+\frac{\gamma_{2}}{2}\right)\left(\arctan\left(Ny\right)-\frac{\pi}{2}\right)\right)}{y^{k+3/2+\beta_{1}+\beta_{2}/2}}dy
\]
\[
=\mathcal{A}_{3}+\mathcal{A}_{4},
\]
say. If $y\in\left(0,1/N\right]$ we obviously have $\arctan\left(Ny\right)-\frac{\pi}{2}\leq-\frac{\pi}{4}$
and so
\begin{align*}
\mathcal{A}_{3}\ll_{k} & \sum_{\rho_{1}:\,\gamma_{1}>0}\gamma_{1}^{\beta_{1}-1/2}\exp\left(-\frac{\pi}{4}\gamma_{1}\right)\sum_{\rho_{2}:\,\gamma_{2}>0}\gamma_{2}^{\beta_{2}/2-1/2}\exp\left(-\frac{\pi}{8}\gamma_{2}\right)\int_{0}^{1/N}N^{k+3/2+\beta_{1}+\beta_{2}/2}dy\ll_{k}N^{k+2}
\end{align*}
 For $\mathcal{A}_{4}$ we can see, following the proof of the Lemma
4, that we have 
\begin{align*}
\mathcal{A}_{4}\ll & N^{k+1/2}\sum_{\rho_{1}:\,\gamma_{1}>0}N^{\beta_{1}}\gamma_{1}^{\beta_{1}-1/2}\sum_{\rho_{2}:\,\gamma_{2}>0}N^{\beta_{2}/2}\gamma_{2}^{\beta_{2}/2-1/2}\int_{1}^{\infty}\frac{\exp\left(-\left(\gamma_{1}+\frac{\gamma_{2}}{2}\right)\arctan\left(\frac{1}{u}\right)\right)}{u^{k+3/2+\beta_{1}+\beta_{2}/2}}du\\
\asymp & N^{k+1/2}\sum_{\rho_{1}:\,\gamma_{1}>0}\sum_{\rho_{2}:\,\gamma_{2}>0}N^{\beta_{1}+\beta_{2}/2}\frac{\gamma_{1}^{\beta_{1}-1/2}\gamma_{2}^{\beta_{2}/2-1/2}}{\left(\gamma_{1}+\frac{\gamma_{2}}{2}\right)^{k+1/2+\beta_{1}+\beta_{2}/2}}
\end{align*}
and observing that 
\[
\frac{\gamma_{1}^{\beta_{1}}\gamma_{2}^{\beta_{2}/2}}{2}\leq\left(\gamma_{1}+\frac{\gamma_{2}}{2}\right)^{\beta_{1}+\beta_{2}/2}
\]
we get 
\begin{align*}
\mathcal{A}_{4}\ll_{k} & N^{k+1/2}\sum_{\rho_{1}:\,\gamma_{1}>0}\sum_{\rho_{2}:\,\gamma_{2}>0}N^{\beta_{1}+\beta_{2}/2}\frac{1}{\gamma_{1}^{1/2}\gamma_{2}^{1/2}\left(\gamma_{1}+\frac{\gamma_{2}}{2}\right)^{k+1/2}}\\
\ll_{k} & N^{k+1/2}\sum_{\rho_{1}:\,\gamma_{1}>0}\frac{1}{\gamma_{1}^{k+1}}\sum_{\rho_{2}:\,0<\gamma_{2}\leq\gamma_{1}}\frac{1}{\gamma_{2}^{1/2}}\\
\ll_{k} & N^{k+1/2}\sum_{\rho_{1}:\,\gamma_{1}>0}\frac{\log\left(\gamma_{1}\right)}{\gamma_{1}^{k+1/2}}
\end{align*}
and so we proved the convergence if $k>1/2$ using the Riemann - Von
Mangoldt formula. Let us consider the case $\gamma_{1}>0$, $\gamma_{2}<0$
(and so we will use the symbol $C_{m,n}$) and let $y\leq0$. Using
again (\ref{eq:arctanboun1}) we have to study 
\[
C_{4,2}^{-}\ll\sum_{\rho_{1}:\,\gamma_{1}>0}\gamma_{1}^{\beta_{1}-1/2}\exp\left(-\frac{\pi}{2}\gamma_{1}\right)\sum_{\rho_{2}:\,\gamma_{2}<0}\left|\gamma_{2}\right|^{\beta_{2}/2-1/2}\int_{-\infty}^{0}\frac{\exp\left(\frac{\gamma_{2}\arctan\left(Ny\right)}{2}-\frac{\pi\left|\gamma_{2}\right|}{4}\right)}{\left|z\right|^{k+3/2+\beta_{1}+\beta_{2}/2}}\left|dz\right|
\]
and using Lemma 4, Lemma 5 and the identity $\arctan\left(x\right)+\arctan\left(1/x\right)=-\pi/2,\,x<0$
we have 
\begin{align*}
C_{4,2}^{-}\ll_{k} & N^{k+3}\sum_{\rho_{1}:\,\gamma_{1}>0}\gamma_{1}^{\beta_{1}-1/2}\exp\left(-\frac{\pi}{2}\gamma_{1}\right)\sum_{\rho_{2}:\,\gamma_{2}<0}\left|\gamma_{2}\right|^{\beta_{2}/2-1/2}\exp\left(-\frac{\pi}{8}\left|\gamma_{2}\right|\right)\\
+ & \sum_{\rho_{1}:\,\gamma_{1}>0}\gamma_{1}^{\beta_{1}-1/2}\exp\left(-\frac{\pi}{2}\gamma_{1}\right)\sum_{\rho_{2}:\,\gamma_{2}<0}\left|\gamma_{2}\right|^{\beta_{2}/2-1/2}\int_{-\infty}^{-1/N}\frac{\exp\left(-\frac{\left|\gamma_{2}\right|}{2}\left(\arctan\left(Ny\right)+\frac{\pi}{2}\right)\right)}{\left|y\right|^{k+3/2+\beta_{1}+\beta_{2}/2}}dy\\
\ll_{k} & N^{k+3}+N^{k+2}\sum_{\rho_{1}:\,\gamma_{1}>0}\gamma_{1}^{\beta_{1}-1/2}\exp\left(-\frac{\pi}{2}\gamma_{1}\right)\sum_{\rho_{2}:\,\gamma_{2}<0}\left|\gamma_{2}\right|^{\beta_{2}/2-1/2}\int_{1}^{\infty}\frac{\exp\left(-\frac{\left|\gamma_{2}\right|}{2}\arctan\left(\frac{1}{u}\right)\right)}{u^{k+3/2+\beta_{1}+\beta_{2}/2}}du\ll_{k}N^{k+3}
\end{align*}
for $k>-1/2$. If $y>0$ we have essentially the same situation exchanging
the role of $\gamma_{1}$ and $\gamma_{2}$. So we can switch the
integral with the series and get 
\[
I_{4}=\frac{1}{4\sqrt{\pi}i}\sum_{\rho_{1}}\Gamma\left(\rho_{1}\right)\sum_{\rho_{2}}\Gamma\left(\frac{\rho_{2}}{2}\right)\int_{\left(1/N\right)}e^{Nz}z^{-k-3/2-\rho_{1}-\rho_{2}/2}dz
\]
\[
=\frac{N^{k+1/2}\sqrt{\pi}}{2}\sum_{\rho_{1}}\sum_{\rho_{2}}N^{\rho_{1}+\rho_{2}/2}\frac{\Gamma\left(\rho_{1}\right)\Gamma\left(\frac{\rho_{2}}{2}\right)}{\Gamma\left(k+3/2+\rho_{1}+\rho_{2}/2\right)}.
\]

\section{Evaluation of $I_{5}$}

We have to evaluate 
\[
I_{5}=\frac{1}{8\pi i}\int_{\left(1/N\right)}e^{Nz}z^{-k-2}\sum_{\rho_{1}}z^{-\rho_{1}/2}\Gamma\left(\frac{\rho_{1}}{2}\right)\sum_{\rho_{2}}z^{-\rho_{2}/2}\Gamma\left(\frac{\rho_{2}}{2}\right)dz
\]
and we can see that the argument used in $I_{4}$ works also in this
case since the presence of $\beta_{1}/2$ instead of $\beta_{1}$
does not alter the validity of the proof. So repeating the reasoning
we can obtain the convergence for $k>1/2$ and so 
\[
I_{5}=\frac{1}{8\pi i}\sum_{\rho_{1}}\Gamma\left(\frac{\rho_{1}}{2}\right)\sum_{\rho_{2}}\Gamma\left(\frac{\rho_{2}}{2}\right)\int_{\left(1/N\right)}e^{Nz}z^{-k-2-\rho_{1}/2-\rho_{2}/2}dz
\]
\[
=\frac{N^{k+1}}{4}\sum_{\rho_{1}}\sum_{\rho_{2}}N^{\rho_{1}/2+\rho_{2}/2}\frac{\Gamma\left(\frac{\rho_{1}}{2}\right)\Gamma\left(\frac{\rho_{2}}{2}\right)}{\Gamma\left(k+2+\rho_{1}/2+\rho_{2}/2\right)}
\]
then $I_{4}+I_{5}=M_{3}\left(N,k\right)$.

\section{Evaluation of $I_{6}$}

We have to evaluate 
\[
I_{6}=\frac{1}{8\pi i}\int_{\left(1/N\right)}e^{Nz}z^{-k-1}\sum_{\rho_{1}}z^{-\rho_{1}}\Gamma\left(\rho_{1}\right)\sum_{\rho_{2}}z^{-\rho_{2}/2}\Gamma\left(\frac{\rho_{2}}{2}\right)\sum_{\rho_{3}}z^{-\rho_{3}/2}\Gamma\left(\frac{\rho_{3}}{2}\right)dz.
\]
We want to switch the integral with three series, so we will prove
the absolute convergence of
\[
A_{6,1}:=\sum_{\rho_{1}}\left|\Gamma\left(\rho_{1}\right)\right|\int_{\left(1/N\right)}\left|e^{Nz}\right|\left|z^{-k-1}\right|\left|z^{-\rho_{1}}\right|\left|\sum_{\rho_{2}}z^{-\rho_{2}/2}\Gamma\left(\frac{\rho_{2}}{2}\right)\right|\left|\sum_{\rho_{3}}z^{-\rho_{3}/2}\Gamma\left(\frac{\rho_{3}}{2}\right)\right|\left|dz\right|,
\]
\[
A_{6,2}:=\sum_{\rho_{1}}\left|\Gamma\left(\rho_{1}\right)\right|\sum_{\rho_{2}}\left|\Gamma\left(\frac{\rho_{2}}{2}\right)\right|\int_{\left(1/N\right)}\left|e^{Nz}\right|\left|z^{-k-1}\right|\left|z^{-\rho_{1}}\right|\left|z^{-\rho_{2}/2}\right|\left|\sum_{\rho_{3}}z^{-\rho_{3}/2}\Gamma\left(\frac{\rho_{3}}{2}\right)\right|\left|dz\right|
\]
and
\[
A_{6,3}:=\sum_{\rho_{1}}\left|\Gamma\left(\rho_{1}\right)\right|\sum_{\rho_{2}}\left|\Gamma\left(\frac{\rho_{2}}{2}\right)\right|\sum_{\rho_{3}}\left|\Gamma\left(\frac{\rho_{3}}{2}\right)\right|\int_{\left(1/N\right)}\left|e^{Nz}\right|\left|z^{-k-1}\right|\left|z^{-\rho_{1}}\right|\left|z^{-\rho_{2}/2}\right|\left|z^{-\rho_{3}/2}\right|\left|dz\right|.
\]
Let us consider
\[
A_{6,1}=\sum_{\rho_{1}}\left|\Gamma\left(\rho_{1}\right)\right|\int_{\left(1/N\right)}\left|e^{Nz}\right|\left|z^{-k-1}\right|\left|z^{-\rho_{1}}\right|\left|\sum_{\rho_{2}}z^{-\rho_{2}/2}\Gamma\left(\frac{\rho_{2}}{2}\right)\right|\left|\sum_{\rho_{3}}z^{-\rho_{3}/2}\Gamma\left(\frac{\rho_{3}}{2}\right)\right|\left|dz\right|,
\]
and we assume, by symmetry, that $\gamma_{1}>0$. Let $y\leq0.$ From
(\ref{eq:arctanboun1}) and recalling the notation (\ref{eq:simbol-})
we have that 
\[
A_{6,1}^{-}\ll N^{k+3}\sum_{\rho_{1}:\,\gamma_{1}>0}N^{\beta_{1}}\gamma_{1}^{\beta_{1}-1/2}\exp\left(-\frac{\pi}{2}\gamma_{1}\right)\int_{-1/N}^{0}\exp\left(\gamma_{1}\arctan\left(N\left|y\right|\right)\right)dy
\]
\[
+\sum_{\rho_{1}:\,\gamma_{1}>0}\gamma_{1}^{\beta_{1}-1/2}\exp\left(-\frac{\pi}{2}\gamma_{1}\right)\int_{-\infty}^{-1/N}\left|z\right|^{-k-1-\beta_{1}}\exp\left(\gamma_{1}\arctan\left(N\left|y\right|\right)\right)\left(N+\left|z\right|^{1/2}\log^{2}\left(2N\left|y\right|\right)\right)^{2}dy
\]
which is bounded by
\[
A_{6,1}^{-}\ll N^{k+3}+N^{2}\sum_{\rho_{1}:\,\gamma_{1}>0}\gamma_{1}^{\beta_{1}-1/2}\exp\left(-\frac{\pi}{2}\gamma_{1}\right)\int_{-\infty}^{-1/N}\left|y\right|^{-k-1-\beta_{1}}dy
\]
\[
+2N\sum_{\rho_{1}:\,\gamma_{1}>0}\gamma_{1}^{\beta_{1}-1/2}\exp\left(-\frac{\pi}{2}\gamma_{1}\right)\int_{-\infty}^{-1/N}\left|y\right|^{-k-1/2-\beta_{1}}\log^{2}\left(2N\left|y\right|\right)dy
\]
\[
+\sum_{\rho_{1}:\,\gamma_{1}>0}\gamma_{1}^{\beta_{1}-1/2}\exp\left(-\frac{\pi}{2}\gamma_{1}\right)\int_{-\infty}^{-1/N}\left|y\right|^{-k-\beta_{1}}\log^{4}\left(2N\left|y\right|\right)dy\ll N^{k+3}
\]
for $k>1$. Let $y>0$. We have 
\[
A_{6,1}^{+}\ll N^{2}\sum_{\rho_{1}:\,\gamma_{1}>0}\gamma_{1}^{\beta_{1}-1/2}\int_{0}^{1/N}\exp\left(\gamma_{1}\arctan\left(Ny\right)-\frac{\pi}{2}\gamma_{1}\right)\frac{\left|dz\right|}{\left|z\right|^{k+1+\beta_{1}}}
\]
\[
+\sum_{\rho_{1}:\,\gamma_{1}>0}\gamma_{1}^{\beta_{1}-1/2}\int_{1/N}^{\infty}\left|z\right|^{-k-1-\beta_{1}}\exp\left(\gamma_{1}\arctan\left(Ny\right)-\frac{\pi}{2}\gamma_{1}\right)\left(N+\left|z\right|^{1/2}\log^{2}\left(2N\left|y\right|\right)\right)^{2}dy.
\]
From Lemma 5 we have 
\[
\sum_{\rho_{1}:\,\gamma_{1}>0}\gamma_{1}^{\beta_{1}-1/2}\int_{0}^{1/N}\exp\left(\gamma_{1}\arctan\left(Ny\right)-\frac{\pi}{2}\gamma_{1}\right)\frac{\left|dz\right|}{\left|z\right|^{k+1+\beta_{1}}}\ll_{k}N^{k+1}
\]
for $k>0$ so 
\[
A_{6,1}^{+}\ll N^{k+3}+\sum_{\rho_{1}:\,\gamma_{1}>0}\gamma_{1}^{\beta_{1}-1/2}\int_{1/N}^{\infty}\left|z\right|^{-k-1-\beta_{1}}\exp\left(\gamma_{1}\arctan\left(Ny\right)-\frac{\pi}{2}\gamma_{1}\right)\left(N+\left|z\right|^{1/2}\log\left(2N\left|y\right|\right)\right)^{2}dy
\]

\[
\ll N^{k+3}+N^{2}\sum_{\rho_{1}:\,\gamma_{1}>0}\gamma_{1}^{\beta_{1}-1/2}\int_{1/N}^{\infty}\exp\left(\gamma_{1}\arctan\left(Ny\right)-\frac{\pi}{2}\gamma_{1}\right)y^{-k-1-\beta_{1}}dy
\]
\[
+2N\sum_{\rho_{1}:\,\gamma_{1}>0}\gamma_{1}^{\beta_{1}-1/2}\int_{1/N}^{\infty}\exp\left(\gamma_{1}\arctan\left(Ny\right)-\frac{\pi}{2}\gamma_{1}\right)\frac{\log^{2}\left(2Ny\right)}{y^{k+1/2+\beta_{1}}}dy
\]
\[
+\sum_{\rho_{1}:\,\gamma_{1}>0}\gamma_{1}^{\beta_{1}-1/2}\int_{1/N}^{\infty}\exp\left(\gamma_{1}\arctan\left(Ny\right)-\frac{\pi}{2}\gamma_{1}\right)\frac{\log^{4}\left(2Ny\right)}{y^{k+\beta_{1}}}dy
\]
and using the well known identity 
\begin{equation}
\arctan\left(x\right)-\frac{\pi}{2}=-\arctan\left(\frac{1}{x}\right),\,x>0\label{eq:arctanboun2}
\end{equation}
 and placing $Ny=u$ we get 
\[
A_{6,1}^{+}\ll N^{k+3}+N^{k+2}\sum_{\rho_{1}:\,\gamma_{1}>0}N^{\beta_{1}}\gamma_{1}^{\beta_{1}-1/2}\int_{1}^{\infty}\exp\left(-\gamma_{1}\arctan\left(\frac{1}{u}\right)\right)u^{-k-1-\beta_{1}}dy
\]
\[
+2N^{k+1/2}\sum_{\rho_{1}:\,\gamma_{1}>0}N^{\beta_{1}}\gamma_{1}^{\beta_{1}-1/2}\int_{1}^{\infty}\exp\left(-\gamma_{1}\arctan\left(\frac{1}{u}\right)\right)\frac{\log^{2}\left(2u\right)}{u^{k+1/2+\beta_{1}}}dy
\]
\[
+N^{k-1}\sum_{\rho_{1}:\,\gamma_{1}>0}N^{\beta_{1}}\gamma_{1}^{\beta_{1}-1/2}\int_{1}^{\infty}\exp\left(-\gamma_{1}\arctan\left(\frac{1}{u}\right)\right)\frac{\log^{4}\left(2u\right)}{u^{k+\beta_{1}}}dy\ll_{k}N^{k+3}
\]
from Lemma 4, assuming $k>3/2$. Now we have to study 
\[
A_{6,2}=\sum_{\rho_{1}}\left|\Gamma\left(\rho_{1}\right)\right|\sum_{\rho_{2}}\left|\Gamma\left(\frac{\rho_{2}}{2}\right)\right|\int_{\left(1/N\right)}\left|e^{Nz}\right|\left|z^{-k-1}\right|\left|z^{-\rho_{1}}\right|\left|z^{-\rho_{2}/2}\right|\left|\sum_{\rho_{3}}z^{-\rho_{3}/2}\Gamma\left(\frac{\rho_{3}}{2}\right)\right|\left|dz\right|
\]
and, by symmetry, we can consider the cases $\gamma_{1},\gamma_{2}>0$
or $\gamma_{1}>0$, $\gamma_{2}<0$. Let $\gamma_{1},\gamma_{2}>0$
and $y\leq0$. From (\ref{eq:arctanboun1}) we have 
\[
B_{6,2}^{-}\ll N\sum_{\rho_{1}:\,\gamma_{1}>0}\gamma_{1}^{\beta_{1}-1/2}\exp\left(-\frac{\pi}{2}\gamma_{1}\right)\sum_{\rho_{2}:\,\gamma_{2}>0}\gamma_{2}^{\beta_{2}/2-1/2}\exp\left(-\frac{\pi}{4}\gamma_{2}\right)\int_{-1/N}^{0}\frac{\left|dz\right|}{\left|z\right|^{k+1+\beta_{1}+\beta_{2}/2}}
\]
\[
+\sum_{\rho_{1}:\,\gamma_{1}>0}\gamma_{1}^{\beta_{1}-1/2}\exp\left(-\frac{\pi}{2}\gamma_{1}\right)\sum_{\rho_{2}:\,\gamma_{2}>0}\gamma_{2}^{\beta_{2}/2-1/2}\exp\left(-\frac{\pi}{4}\gamma_{2}\right)\int_{-\infty}^{-1/N}\frac{N+\left|y\right|^{1/2}\log^{2}\left(2N\left|y\right|\right)}{\left|y\right|^{k+1+\beta_{1}+\beta_{2}/2}}d\left|y\right|
\]
\[
\ll N^{k+3}+\sum_{\rho_{1}:\,\gamma_{1}>0}\gamma_{1}^{\beta_{1}-1/2}\exp\left(-\frac{\pi}{2}\gamma_{1}\right)\sum_{\rho_{2}:\,\gamma_{2}>0}\gamma_{2}^{\beta_{2}/2-1/2}\exp\left(-\frac{\pi}{4}\gamma_{2}\right)\int_{-\infty}^{-1/N}\frac{1}{\left|y\right|^{k+1+\beta_{1}+\beta_{2}/2}}d\left|y\right|
\]
\[
+\sum_{\rho_{1}:\,\gamma_{1}>0}\gamma_{1}^{\beta_{1}-1/2}\exp\left(-\frac{\pi}{2}\gamma_{1}\right)\sum_{\rho_{2}:\,\gamma_{2}>0}\gamma_{2}^{\beta_{2}/2-1/2}\exp\left(-\frac{\pi}{4}\gamma_{2}\right)\int_{-\infty}^{-1/N}\frac{\log^{2}\left(2N\left|y\right|\right)}{\left|y\right|^{k+1/2+\beta_{1}+\beta_{2}/2}}d\left|y\right|\ll_{k}N^{k+3}
\]
for $k>1/2$. Let $y>0,$ and so the symbol $B_{m,n}$. We recall
again that we have to split the integral for $y\in\left(0,1/N\right)$
and $y\in\left(1/N,\,\infty\right)$ since, by (\ref{eq:modz-1})
and (\ref{eq:cambioeval}), we have different estimation in these
two set. We have that
\[
B_{6,2}^{+}\ll N\sum_{\rho_{1}:\,\gamma_{1}>0}\gamma_{1}^{\beta_{1}-1/2}\sum_{\rho_{2}:\,\gamma_{2}>0}\gamma_{2}^{\beta_{2}/2-1/2}\int_{0}^{1/N}\exp\left(\left(\gamma_{1}+\frac{\gamma_{2}}{2}\right)\left(\arctan\left(Ny\right)-\frac{\pi}{2}\right)\right)\frac{\left|dz\right|}{\left|z\right|^{k+1+\beta_{1}+\beta_{2}/2}}
\]
\[
+\sum_{\rho_{1}:\,\gamma_{1}>0}\gamma_{1}^{\beta_{1}-1/2}\sum_{\rho_{2}:\,\gamma_{2}>0}\gamma_{2}^{\beta_{2}/2-1/2}\int_{1/N}^{\infty}\exp\left(\left(\gamma_{1}+\frac{\gamma_{2}}{2}\right)\left(\arctan\left(Ny\right)-\frac{\pi}{2}\right)\right)\frac{\left(N+y^{1/2}\log^{2}\left(2Ny\right)\right)}{y^{k+1+\beta_{1}+\beta_{2}/2}}dy
\]
\[
\ll N\sum_{\rho_{1}:\,\gamma_{1}>0}\gamma_{1}^{\beta_{1}-1/2}\exp\left(-\frac{\gamma_{1}}{4}\right)\sum_{\rho_{2}:\,\gamma_{2}>0}\gamma_{2}^{\beta_{2}/2-1/2}\exp\left(-\frac{\gamma_{2}}{8}\right)\int_{0}^{1/N}N^{k+1+\beta_{1}+\beta_{2}/2}dy
\]
\[
\sum_{\rho_{1}:\,\gamma_{1}>0}\gamma_{1}^{\beta_{1}-1/2}\sum_{\rho_{2}:\,\gamma_{2}>0}\gamma_{2}^{\beta_{2}/2-1/2}\int_{1/N}^{\infty}\exp\left(\left(\gamma_{1}+\frac{\gamma_{2}}{2}\right)\left(\arctan\left(Ny\right)-\frac{\pi}{2}\right)\right)\frac{\left(N+y^{1/2}\log^{2}\left(2Ny\right)\right)}{y^{k+1+\beta_{1}+\beta_{2}/2}}dy
\]
which is bounded by
\[
B_{6,2}^{+}\ll N^{k+3}+N\sum_{\rho_{1}:\,\gamma_{1}>0}\gamma_{1}^{\beta_{1}-1/2}\sum_{\rho_{2}:\,\gamma_{2}>0}\gamma_{2}^{\beta_{2}/2-1/2}\int_{1/N}^{\infty}\frac{\exp\left(\left(\gamma_{1}+\frac{\gamma_{2}}{2}\right)\left(\arctan\left(Ny\right)-\frac{\pi}{2}\right)\right)}{y^{k+1+\beta_{1}+\beta_{2}/2}}dy
\]
\[
+\sum_{\rho_{1}:\,\gamma_{1}>0}\gamma_{1}^{\beta_{1}-1/2}\sum_{\rho_{2}:\,\gamma_{2}>0}\gamma_{2}^{\beta_{2}/2-1/2}\int_{1/N}^{\infty}\frac{\exp\left(\left(\gamma_{1}+\frac{\gamma_{2}}{2}\right)\left(\arctan\left(Ny\right)-\frac{\pi}{2}\right)\right)\log^{2}\left(2Ny\right)}{y^{k+1/2+\beta_{1}+\beta_{2}/2}}dy
\]
and again from (\ref{eq:arctanboun2}) and placing $Ny=u$ we get
\[
B_{6,2}^{+}\ll N^{k+3}\sum_{\rho_{1}:\,\gamma_{1}>0}\gamma_{1}^{\beta_{1}-1/2}\sum_{\rho_{2}:\,\gamma_{2}>0}\gamma_{2}^{\beta_{2}/2-1/2}\int_{1}^{\infty}\frac{\exp\left(-\left(\gamma_{1}+\frac{\gamma_{2}}{2}\right)\arctan\left(\frac{1}{u}\right)\right)}{u^{k+1+\beta_{1}+\beta_{2}/2}}du
\]
\[
+N^{k+3/2}\sum_{\rho_{1}:\,\gamma_{1}>0}\gamma_{1}^{\beta_{1}-1/2}\sum_{\rho_{2}:\,\gamma_{2}>0}\gamma_{2}^{\beta_{2}/2-1/2}\int_{1}^{\infty}\frac{\exp\left(-\left(\gamma_{1}+\frac{\gamma_{2}}{2}\right)\arctan\left(\frac{1}{u}\right)\right)\log^{2}\left(2u\right)}{u^{k+1/2+\beta_{1}+\beta_{2}/2}}du
\]
and from the proof of Lemma 4 we have 
\[
B_{6,2}^{+}\ll_{k}N^{k+3}\sum_{\rho_{1}:\,\gamma_{1}>0}\gamma_{1}^{\beta_{1}-1/2}\sum_{\rho_{2}:\,\gamma_{2}>0}\gamma_{2}^{\beta_{2}/2-1/2}\left(\gamma_{1}+\frac{\gamma_{2}}{2}\right)^{-k+1/2-\beta_{1}-\beta_{2}/2}
\]
and observing that 
\[
\gamma_{1}^{\beta_{1}}\left(\frac{\gamma_{2}}{2}\right)^{\beta_{2}/2}\leq\left(\gamma_{1}+\frac{\gamma_{2}}{2}\right)^{\beta_{1}}\left(\gamma_{1}+\frac{\gamma_{2}}{2}\right)^{\beta_{2}/2}=\left(\gamma_{1}+\frac{\gamma_{2}}{2}\right)^{\beta_{1}/2+\beta_{2}/2}
\]
we get 
\[
B_{6,2}^{+}\ll_{k}N^{k+3}\sum_{\rho_{1}:\,\gamma_{1}>0}\sum_{\rho_{2}:\,\gamma_{2}>0}\frac{1}{\gamma_{1}^{1/2}\gamma_{2}^{1/2}\left(\gamma_{1}+\gamma_{2}\right)^{k-1/2}}
\]
\[
\ll_{k}N^{k+3}\sum_{\rho_{1}:\,\gamma_{1}>0}\frac{1}{\gamma_{1}^{k}}\sum_{\rho_{2}:\,0<\gamma_{2}\leq\gamma_{1}}\frac{1}{\gamma_{2}^{1/2}}
\]
\[
\ll_{k}N^{k+3}\sum_{\rho_{1}:\,\gamma_{1}>0}\frac{\log\left(\gamma_{1}\right)}{\gamma_{1}^{k-1/2}}
\]
and so the convergence if $k>3/2$. Let us assume that $\gamma_{1}>0$,
$\gamma_{2}<0$ and $y\le0$. From (\ref{eq:arctanboun1}) we have
\[
C_{6,2}^{-}\ll\sum_{\rho_{1}:\,\gamma_{1}>0}\gamma_{1}^{\beta_{1}-1/2}\exp\left(-\frac{\pi}{2}\gamma_{1}\right)\sum_{\rho_{2}:\,\gamma_{2}<0}\left|\gamma_{2}\right|^{\beta_{2}/2-1/2}\int_{-1/N}^{0}\frac{\exp\left(-\frac{\left|\gamma_{2}\right|}{2}\left(\arctan\left(Ny\right)+\frac{\pi}{2}\right)\right)\left|dz\right|}{\left|z\right|^{k+1+\beta_{1}+\beta_{2}/2}}
\]
\[
+N\sum_{\rho_{1}:\,\gamma_{1}>0}\gamma_{1}^{\beta_{1}-1/2}\exp\left(-\frac{\pi}{2}\gamma_{1}\right)\sum_{\rho_{2}:\,\gamma_{2}<0}\left|\gamma_{2}\right|^{\beta_{2}/2-1/2}\int_{-\infty}^{-1/N}\frac{\exp\left(-\frac{\left|\gamma_{2}\right|}{2}\left(\arctan\left(Ny\right)+\frac{\pi}{2}\right)\right)\left|dz\right|}{\left|z\right|^{k+1+\beta_{1}+\beta_{2}/2}}
\]
\[
+\sum_{\rho_{1}:\,\gamma_{1}>0}\gamma_{1}^{\beta_{1}-1/2}\exp\left(-\frac{\pi}{2}\gamma_{1}\right)\sum_{\rho_{2}:\,\gamma_{2}<0}\left|\gamma_{2}\right|^{\beta_{2}/2-1/2}\int_{-\infty}^{-1/N}\frac{\exp\left(-\frac{\left|\gamma_{2}\right|}{2}\left(\arctan\left(Ny\right)+\frac{\pi}{2}\right)\right)\log^{2}\left(2N\left|y\right|\right)\left|dz\right|}{\left|z\right|^{k+1/2+\beta_{1}+\beta_{2}/2}}
\]
hence
\[
C_{6,2}^{-}\ll N^{k+3/2}\sum_{\rho_{1}:\,\gamma_{1}>0}\gamma_{1}^{\beta_{1}-1/2}\exp\left(-\frac{\pi}{2}\gamma_{1}\right)\sum_{\rho_{2}:\,\gamma_{2}<0}\left|\gamma_{2}\right|^{\beta_{2}/2-1/2}\exp\left(-\frac{\pi}{8}\left|\gamma_{2}\right|\right)
\]
\[
+N^{k+5/2}\sum_{\rho_{1}:\,\gamma_{1}>0}\gamma_{1}^{\beta_{1}-1/2}\exp\left(-\frac{\pi}{2}\gamma_{1}\right)\sum_{\rho_{2}}\left|\gamma_{2}\right|^{\beta_{2}/2-1/2}\int_{1}^{\infty}\frac{\exp\left(-\frac{\left|\gamma_{2}\right|}{2}\arctan\left(\frac{1}{u}\right)\right)}{u^{k+1+\beta_{1}+\beta_{2}/2}}du
\]
\[
+N^{k+2}\sum_{\rho_{1}:\,\gamma_{1}>0}\gamma_{1}^{\beta_{1}-1/2}\exp\left(-\frac{\pi}{2}\gamma_{1}\right)\sum_{\rho_{2}:\,\gamma_{2}<0}\left|\gamma_{2}\right|^{\beta_{2}/2-1/2}\int_{1}^{\infty}\frac{\exp\left(-\frac{\left|\gamma_{2}\right|}{2}\arctan\left(\frac{1}{u}\right)\right)\log^{2}\left(2u\right)}{u^{k+1/2+\beta_{1}+\beta_{2}/2}}du\ll N^{k+5/2}
\]
for $k>1.$ If $y>0$ we have essentially the same calculations exchanging
the role of $\gamma_{1}$ and $\gamma_{2}$. So we have to consider
\[
A_{6,3}=\sum_{\rho_{1}}\left|\Gamma\left(\rho_{1}\right)\right|\sum_{\rho_{2}}\left|\Gamma\left(\frac{\rho_{2}}{2}\right)\right|\sum_{\rho_{3}}\left|\Gamma\left(\frac{\rho_{3}}{2}\right)\right|\int_{\left(1/N\right)}\left|e^{Nz}\right|\left|z^{-k-1}\right|\left|z^{-\rho_{1}}\right|\left|z^{-\rho_{2}/2}\right|\left|z^{-\rho_{3}/2}\right|\left|dz\right|.
\]
It is sufficient to consider the cases $\gamma_{i}>0,\,i=1,2,3$,
$\gamma_{1},\gamma_{2}>0$ and $\gamma_{3}<0$ and lastly $\gamma_{1}>0,$
$\gamma_{2},\gamma_{3}<0$. We will use the symbol $D_{6,3}$ when
we consider $A_{6,3}$ with the assumption $\gamma_{i}>0,\,i=1,2,3$,
the symbol $E_{6,3}$ when we consider $A_{6,3}$ with the assumption
$\gamma_{1},\gamma_{2}>0$ and $\gamma_{3}<0$ and $F_{6,3}$ when
we consider $A_{6,3}$ with the assumption $\gamma_{1}>0,$ $\gamma_{2},\gamma_{3}<0$.
From (\ref{eq:arctanboun1}) we have 
\[
D_{6,3}^{-}\ll\sum_{\rho_{1}:\,\gamma_{1}>0}\gamma_{1}^{\beta_{1}-1/2}\exp\left(-\frac{\pi}{2}\gamma_{1}\right)\sum_{\rho_{2}:\,\gamma_{2}>0}\gamma_{2}^{\beta_{2}/2-1/2}\exp\left(-\frac{\pi}{4}\gamma_{2}\right)\sum_{\rho_{3}:\,\gamma_{3}>0}\gamma_{3}^{\beta_{3}/2-1/2}\exp\left(-\frac{\pi}{4}\gamma_{3}\right)
\]
\[
\cdot\int_{-1/N}^{0}N^{k+1+\beta_{1}+\beta_{2}/2+\beta_{3}/2}dy
\]
\[
+\sum_{\rho_{1}:\,\gamma_{1}>0}\gamma_{1}^{\beta_{1}-1/2}\exp\left(-\frac{\pi}{2}\gamma_{1}\right)\sum_{\rho_{2}:\,\gamma_{2}>0}\gamma_{2}^{\beta_{2}/2-1/2}\exp\left(-\frac{\pi}{4}\gamma_{2}\right)\sum_{\rho_{3}:\,\gamma_{3}>0}\gamma_{3}^{\beta_{3}/2-1/2}\exp\left(-\frac{\pi}{4}\gamma_{3}\right)
\]
\[
\cdot\int_{-\infty}^{-1/N}\left|y\right|^{-k-1-\beta_{1}-\beta_{2}/2-\beta_{3}/2}dy\ll N^{k+2}
\]

for $k>1$. Let $y>0$. We have 
\[
D_{6,3}^{+}\ll\sum_{\rho_{1}:\,\gamma_{1}>0}\gamma_{1}^{\beta_{1}-1/2}\sum_{\rho_{2}:\,\gamma_{2}>0}\gamma_{2}^{\beta_{2}/2-1/2}\sum_{\rho_{3}:\,\gamma_{3}>0}\gamma_{3}^{\beta_{3}/2-1/2}
\]
\[
\cdot\int_{0}^{1/N}\exp\left(\left(\gamma_{1}+\frac{\gamma_{2}}{2}+\frac{\gamma_{3}}{2}\right)\left(\arctan\left(Ny\right)-\frac{\pi}{2}\right)\right)N^{k+1+\beta_{1}+\beta_{2}/2+\beta_{3}/2}dy
\]
\[
+\sum_{\rho_{1}:\,\gamma_{1}>0}\gamma_{1}^{\beta_{1}-1/2}\sum_{\rho_{2}:\,\gamma_{2}>0}\gamma_{2}^{\beta_{2}/2-1/2}\sum_{\rho_{3}:\,\gamma_{3}>0}\gamma_{3}^{\beta_{3}/2-1/2}\int_{1/N}^{\infty}\frac{\exp\left(\left(\gamma_{1}+\frac{\gamma_{2}}{2}+\frac{\gamma_{3}}{2}\right)\left(\arctan\left(Ny\right)-\frac{\pi}{2}\right)\right)}{y^{k+1+\beta_{1}+\beta_{2}/2+\beta_{3}/2}}dy
\]
\[
\ll N^{k+2}\sum_{\rho_{1}:\,\gamma_{1}>0}\gamma_{1}^{\beta_{1}-1/2}\exp\left(-\frac{\pi\gamma_{1}}{4}\right)\sum_{\rho_{2}:\,\gamma_{2}>0}\gamma_{2}^{\beta_{2}/2-1/2}\exp\left(-\frac{\pi\gamma_{2}}{8}\right)\sum_{\rho_{3}:\,\gamma_{3}>0}\gamma_{3}^{\beta_{3}/2-1/2}\exp\left(-\frac{\pi\gamma_{2}}{8}\right)
\]
\[
+N^{k+2}\sum_{\rho_{1}:\,\gamma_{1}>0}\gamma_{1}^{\beta_{1}-1/2}\sum_{\rho_{2}:\,\gamma_{2}>0}\gamma_{2}^{\beta_{2}/2-1/2}\sum_{\rho_{3}:\,\gamma_{3}>0}\gamma_{3}^{\beta_{3}/2-1/2}\int_{1}^{\infty}\frac{\exp\left(-\left(\gamma_{1}+\frac{\gamma_{2}}{2}+\frac{\gamma_{3}}{2}\right)\arctan\left(\frac{1}{u}\right)\right)}{u^{k+1+\beta_{1}+\beta_{2}/2+\beta_{3}/2}}du
\]
and from the proof of the Lemma 4 we get 
\[
D_{6,3}^{+}\ll N^{k+2}+N^{k+2}\sum_{\rho_{1}:\,\gamma_{1}>0}\gamma_{1}^{\beta_{1}-1/2}\sum_{\rho_{2}:\,\gamma_{2}>0}\gamma_{2}^{\beta_{2}/2-1/2}\sum_{\rho_{3}:\,\gamma_{3}>0}\gamma_{3}^{\beta_{3}/2-1/2}\left(\gamma_{1}+\frac{\gamma_{2}}{2}+\frac{\gamma_{3}}{2}\right)^{-k-\beta_{1}-\beta_{2}/2-\beta_{3}/2}
\]
and observing that 
\[
\frac{\gamma_{1}^{\beta_{1}}\gamma_{2}^{\beta_{2}/2}\gamma_{3}^{\beta_{3}/2}}{2}\leq\left(\gamma_{1}+\frac{\gamma_{2}}{2}+\frac{\gamma_{3}}{2}\right)^{\beta_{1}+\beta_{2}/2+\beta_{3}/2}
\]
we get 
\[
D_{6,3}^{+}\ll_{k}N^{k+2}+N^{k+2}\sum_{\rho_{1}:\,\gamma_{1}>0}\sum_{\rho_{2}:\,\gamma_{2}>0}\sum_{\rho_{3}:\,\gamma_{3}>0}\frac{1}{\gamma_{1}^{1/2}\gamma_{2}^{1/2}\gamma_{3}^{1/2}\left(\gamma_{1}+\frac{\gamma_{2}}{2}+\frac{\gamma_{3}}{2}\right)^{k}}
\]
and from AM-GM inequality we get 
\[
D_{6,3}^{+}\ll N^{k+2}+N^{k+2}\sum_{\rho_{1}:\,\gamma_{1}>0}\frac{1}{\gamma_{1}^{k/3+1/2}}\sum_{\rho_{2}:\,\gamma_{2}>0}\frac{1}{\gamma_{2}^{k/3+1/2}}\sum_{\rho_{3}:\,\gamma_{3}>0}\frac{1}{\gamma_{3}^{k/3+1/2}}\ll N^{k+2}
\]
for $k>3/2$. Let $\gamma_{1},\gamma_{2}>0,\,\gamma_{3}<0$ (and so
the symbol $E_{m,n}$) and $y\leq0$. From (\ref{eq:arctanboun1})
we have 
\[
E_{6,3}^{-}\ll\sum_{\rho_{1}:\,\gamma_{1}>0}\gamma_{1}^{\beta_{1}-1/2}\exp\left(-\frac{\pi}{2}\gamma_{1}\right)\sum_{\rho_{2}:\,\gamma_{2}>0}\gamma_{2}^{\beta_{2}/2-1/2}\exp\left(-\frac{\pi}{4}\gamma_{2}\right)\sum_{\rho_{3}:\,\gamma_{3}<0}\left|\gamma_{3}\right|^{\beta_{3}/2-1/2}
\]
\[
\cdot\int_{-1/N}^{0}\frac{\exp\left(-\frac{\left|\gamma_{3}\right|}{2}\left(\arctan\left(Ny\right)+\frac{\pi}{2}\right)\right)}{\left|z\right|^{k+1+\beta_{1}+\beta_{2}/2+\beta_{3}/2}}\left|dz\right|
\]
\[
+\sum_{\rho_{1}:\,\gamma_{1}>0}\gamma_{1}^{\beta_{1}-1/2}\exp\left(-\frac{\pi}{2}\gamma_{1}\right)\sum_{\rho_{2}:\,\gamma_{2}>0}\gamma_{2}^{\beta_{2}/2-1/2}\exp\left(-\frac{\pi}{4}\gamma_{2}\right)\sum_{\rho_{3}:\,\gamma_{3}<0}\left|\gamma_{3}\right|^{\beta_{3}/2-1/2}
\]
\[
\cdot\int_{-\infty}^{-1/N}\frac{\exp\left(-\frac{\left|\gamma_{3}\right|}{2}\left(\arctan\left(Ny\right)+\frac{\pi}{2}\right)\right)}{\left|z\right|^{k+1+\beta_{1}+\beta_{2}/2+\beta_{3}/2}}\left|dz\right|
\]
\[
\ll N^{k+2}+N^{k+2}\sum_{\rho_{1}:\,\gamma_{1}>0}\gamma_{1}^{\beta_{1}-1/2}\exp\left(-\frac{\pi}{2}\gamma_{1}\right)\sum_{\rho_{2}:\,\gamma_{2}>0}\gamma_{2}^{\beta_{2}/2-1/2}\exp\left(-\frac{\pi}{4}\gamma_{2}\right)\sum_{\rho_{3}:\,\gamma_{3}<0}\left|\gamma_{3}\right|^{\beta_{3}/2-1/2}
\]
\[
\cdot\int_{1}^{\infty}\frac{\exp\left(-\frac{\left|\gamma_{3}\right|}{2}\arctan\left(\frac{1}{u}\right)\right)}{u^{k+1+\beta_{1}+\beta_{2}/2+\beta_{3}/2}}du\ll N^{k+2}
\]
from the proof of Lemma 4, for $k>1/2$. If $y>0$ we have essentially
the same calculations exchanging the role of $\gamma_{2}$ and $\gamma_{3}$.
Let $\gamma_{2},\gamma_{2}<0$ , $\gamma_{1}>0$ and $y<0$. Recalling
(\ref{eq:simbol-}) we have 
\[
F_{6,3}^{-}\ll\sum_{\rho_{1}:\,\gamma_{1}>0}\gamma_{1}^{\beta_{1}-1/2}\exp\left(-\frac{\pi}{2}\gamma_{1}\right)\sum_{\rho_{2}:\,\gamma_{2}<0}\left|\gamma_{2}\right|^{\beta_{2}/2-1/2}\sum_{\rho_{3}:\,\gamma_{3}<0}\left|\gamma_{3}\right|^{\beta_{3}/2-1/2}
\]
\[
\cdot\int_{-1/N}^{0}\frac{\exp\left(-\left(\frac{\left|\gamma_{2}\right|+\left|\gamma_{3}\right|}{2}\right)\left(\arctan\left(Ny\right)+\frac{\pi}{2}\right)\right)}{\left|z\right|^{k+1+\beta_{1}+\beta_{2}/2+\beta_{3}/2}}\left|dz\right|
\]
\[
+\sum_{\rho_{1}:\,\gamma_{1}>0}\gamma_{1}^{\beta_{1}-1/2}\exp\left(-\frac{\pi}{2}\gamma_{1}\right)\sum_{\rho_{2}:\,\gamma_{2}<0}\left|\gamma_{2}\right|^{\beta_{2}/2-1/2}\sum_{\rho_{3}:\,\gamma_{3}<0}\left|\gamma_{3}\right|^{\beta_{3}/2-1/2}
\]
\[
\cdot\int_{-\infty}^{-1/N}\frac{\exp\left(-\left(\frac{\left|\gamma_{2}\right|+\left|\gamma_{3}\right|}{2}\right)\left(\arctan\left(Ny\right)+\frac{\pi}{2}\right)\right)}{\left|z\right|^{k+1+\beta_{1}+\beta_{2}/2+\beta_{3}/2}}\left|dz\right|
\]
\[
\ll N^{k+2}+\sum_{\rho_{1}:\,\gamma_{1}>0}\gamma_{1}^{\beta_{1}-1/2}\exp\left(-\frac{\pi}{2}\gamma_{1}\right)\sum_{\rho_{2}:\,\gamma_{2}<0}\left|\gamma_{2}\right|^{\beta_{2}/2-1/2}\sum_{\rho_{3}:\,\gamma_{3}<0}\left|\gamma_{3}\right|^{\beta_{3}/2-1/2}
\]
\[
\cdot\int_{1}^{\infty}\frac{\exp\left(-\left(\frac{\left|\gamma_{2}\right|+\left|\gamma_{3}\right|}{2}\right)\arctan\left(\frac{1}{u}\right)\right)}{u^{k+1+\beta_{1}+\beta_{2}/2+\beta_{3}/2}}du
\]
\[
\ll N^{k+2}\sum_{\rho_{2}:\,\gamma_{2}<0}\sum_{\rho_{3}:\,\gamma_{3}<0}\left|\gamma_{2}\right|^{\beta_{2}/2-1/2}\left|\gamma_{3}\right|^{\beta_{3}/2-1/2}\left(\left|\gamma_{2}\right|+\left|\gamma_{3}\right|\right)^{-k-\beta_{2}/2-\beta_{3}/2}
\]
\[
\ll N^{k+2}\sum_{\rho_{2}:\,\gamma_{2}<0}\left|\gamma_{2}\right|^{-k-1/2}\sum_{\rho_{3}:\,\gamma_{3}<0}\left|\gamma_{3}\right|^{-k-1/2}\ll N^{k+2}
\]
using Lemma 4, for $k>1/2$. Let $y>0$. Observing that $-\left(\frac{\left|\gamma_{2}\right|+\left|\gamma_{3}\right|}{2}\right)\left(\arctan\left(Ny\right)+\frac{\pi}{2}\right)\leq-\left(\frac{\left|\gamma_{2}\right|+\left|\gamma_{3}\right|}{2}\right)\frac{\pi}{2}$
we have 
\[
F_{6,3}^{+}\ll\sum_{\rho_{1}:\,\gamma_{1}>0}\gamma_{1}^{\beta_{1}-1/2}\exp\left(-\frac{\pi}{4}\gamma_{1}\right)\sum_{\rho_{2}:\,\gamma_{2}<0}\left|\gamma_{2}\right|^{\beta_{2}/2-1/2}\exp\left(-\frac{\pi}{8}\left|\gamma_{2}\right|\right)\sum_{\rho_{3}:\,\gamma_{3}<0}\left|\gamma_{3}\right|^{\beta_{3}/2-1/2}\exp\left(-\frac{\pi}{8}\left|\gamma_{3}\right|\right)
\]
\[
\cdot\int_{0}^{1/N}\frac{\left|dz\right|}{\left|z\right|^{k+1+\beta_{1}+\beta_{2}/2+\beta_{3}/2}}
\]
\[
+\sum_{\rho_{1}:\,\gamma_{1}>0}\gamma_{1}^{\beta_{1}-1/2}\sum_{\rho_{2}:\,\gamma_{2}<0}\left|\gamma_{2}\right|^{\beta_{2}/2-1/2}\exp\left(-\frac{\pi}{8}\left|\gamma_{2}\right|\right)\sum_{\rho_{3}:\,\gamma_{3}<0}\left|\gamma_{3}\right|^{\beta_{3}/2-1/2}\exp\left(-\frac{\pi}{8}\left|\gamma_{3}\right|\right)
\]
\[
\cdot\int_{1/N}^{\infty}\frac{\exp\left(\gamma_{1}\left(\arctan\left(Ny\right)-\frac{\pi}{2}\right)\right)}{\left|z\right|^{k+1+\beta_{1}+\beta_{2}/2+\beta_{3}/2}}dy
\]
\[
\ll N^{k+2}\sum_{\rho_{1}:\,\gamma_{1}>0}\gamma_{1}^{\beta_{1}-1/2}\int_{1/N}^{\infty}\frac{\exp\left(-\gamma_{1}\arctan\left(\frac{1}{u}\right)\right)}{\left|z\right|^{k+1+\beta_{1}}}du\ll N^{k+2}
\]
from Lemma 4 for $k>1/2.$ Now we can exchange the integral with the
series and get 
\[
I_{6}=\frac{1}{8\pi i}\sum_{\rho_{1}}\Gamma\left(\rho_{1}\right)\sum_{\rho_{2}}\Gamma\left(\frac{\rho_{2}}{2}\right)\sum_{\rho_{3}}\Gamma\left(\frac{\rho_{3}}{2}\right)\int_{\left(1/N\right)}e^{Nz}z^{-k-1-\rho_{1}-\rho_{2}/2-\rho_{3}/2}dz
\]
\[
=\frac{N^{k}}{4}\sum_{\rho_{1}}\sum_{\rho_{2}}\sum_{\rho_{3}}\frac{N^{\rho_{1}+\rho_{2}/2+\rho_{3}/2}\Gamma\left(\rho_{1}\right)\Gamma\left(\frac{\rho_{2}}{2}\right)\Gamma\left(\frac{\rho_{3}}{2}\right)}{\Gamma\left(k+\rho_{1}+\rho_{2}/2+\rho_{3}/2\right)}
\]
then $I_{6}=M_{4}\left(N,k\right)$.

$\,$

Marco Cantarini

Universit\`a di Ferrara

Dipartimento di Matematica e Informatica

Via Machiavelli, 30 

44121 Ferrara, Italy

email: cantarini\_m@libero.it
\end{document}